\documentclass[a4paper,11pt]{article}
\usepackage{amsmath,amsthm,amssymb,amsfonts,bm,graphicx,hyperref}
\usepackage[top=3cm, bottom=3cm, left=3cm, right=3cm]{geometry}
\usepackage[shortlabels]{enumitem}
\usepackage{booktabs}

\newcommand{\R}{\mathbb{R}}

\newcommand{\nn}{\ensuremath{\mathcal{N}}}

\newcommand{\cop}{\ensuremath{\mathcal{COP}}}

\def\btheta{\bm\theta}

\newcommand{\plmod}{\ensuremath{\dotplus}}

\newcommand{\1}{\mathbf 1}
\newcommand{\0}{\mathbf 0}

\newtheorem{defin}{Definition}

\newtheorem{theorem}[defin]{Theorem}

\DeclareMathOperator{\diag}{diag}

\title{Corrigendum to ``SPN graphs: when copositive $=$ SPN"\thanks{This work was supported by grant no.\ 2219/15
 by ISF-NSFC joint scientific research program. }
}

\author{
Naomi Shaked-Monderer\thanks{The Max Stern Yezreel Valley College, Yezreel Valley 1930600, Israel.
Email: nomi@technion.ac.il }
}

\begin{document}
\maketitle
\begin{abstract}
\noindent In this corrigendum, an error in the proof of a theorem in [Linear Algebra and its Applications 509 (2016) 82--113] is pointed out.
This theorem states that every graph $T_n$ consisting of $n-2$ triangles sharing a common base is SPN.
An alternative proof is given here for the case $n=5$, but for all $n>5$ it remains open whether $T_n$ is SPN. As a result, the question whether $K_{2,n}$, $n>4$, is SPN
also remains open.
\medskip



\end{abstract}

We  report an error in the proof of  Theorem 6.4 in the paper ``SPN graphs: when copositive $=$ SPN" \cite{Shaked2016}, that we can only partially correct.
The error is in Lemma 6.2, which is incorrect (the error in the proof of this lemma is in the computation of $P/I$ in the 6th row). As a result, also Corollary 6.3 is incorrect, and Theorem 6.4 and the resulting Theorem 9.2 cannot be deduced. At this point it is unclear whether these two theorems are valid for every $n$. We provide
 in this corrigendum an alternative proof to Theorem 6.4 for the case $n=5$, and change Theorem 9.2 accordingly. For terminology and notations, see the original paper \cite{Shaked2016}.

\section*{$T_5$ and $K_{2,4}$ are SPN}
We begin by recalling some known results about the matrices
\[S(\btheta)=\left(\begin {array}{ccccc} 1&-\cos({\theta_1})&\cos({\theta_1+\theta_2})&\cos({\theta_4+\theta_5})&-\cos({\theta_5})\\
\noalign{\medskip}-\cos({\theta_1})&1&-\cos({\theta_2})&\cos({\theta_2+\theta_3})&\cos({\theta_5+\theta_1})\\
\noalign{\medskip}\cos({\theta_1+\theta_2})&-\cos({\theta_2})&1&-\cos(\theta_3)&\cos({\theta_3+\theta_4})\\
\noalign{\medskip}\cos({\theta_4+\theta_5})&\cos({\theta_2+\theta_3})&-\cos({\theta_3})&1&-\cos({\theta_4})\\
\noalign{\medskip}-\cos({\theta_5})&\cos({\theta_5+\theta_1})&\cos({\theta_3+\theta_4})&-\cos({\theta_4})&1\end {array}
\right),\]
where $\btheta\in \R^5_+$ and $\1^T\btheta\le \pi$. It is known that

\begin{enumerate}
\item[{\rm(a)}]  $S(\btheta)$ is copositive.
\item[{\rm(b)}]  $S(\0)$ is the Horn matrix. In particular, it is an exceptional extremal matrix in $\cop_5$.
\item[{\rm(c)}]  If $\1^T\btheta= \pi$, then $S(\btheta)$ is positive semidefinite, or rank $2$.
\item[{\rm(d)}]  If  $\btheta>0$ and $\1^T\btheta< \pi$,  then $S(\btheta)$ is an exceptional extremal matrix, called a Hildebrand matrix.
\item[{\rm(e)}]  If $\btheta\ne \0$ has at least one zero entry  and $\1^T\btheta< \pi$, then $S(\btheta)$ is an exceptional $\nn$-irreducible matrix.
\end{enumerate}
For (a)--(b), (e), see page 1611  in \cite{DickinsonDuerGijbenHildebrand2013a}. For (d) see  \cite{Hildebrand2012}.
For (c), note that in the case that  $\1^T\btheta=\pi$,     %%%% added \1^T
\begin{multline*}S(\btheta)=\\ \small
\left(\begin {array}{ccccc} 1&-\cos({\theta_1})&\cos({\theta_1+\theta_2})&\cos({\theta_4+\theta_5})&-\cos({\theta_5})\\
\noalign{\medskip}-\cos({\theta_1})&1&-\cos({\theta_2})&-\cos({\theta_4+\theta_5+\theta_1})&\cos({\theta_5+\theta_1})\\
\noalign{\medskip}\cos({\theta_1+\theta_2})&-\cos({\theta_2})&1& \cos({\theta_4+\theta_5+\theta_1+\theta_2})&-\cos({\theta_5+\theta_1+\theta_2})\\
\noalign{\medskip}\cos({\theta_4+\theta_5})&-\cos({\theta_4+\theta_5+\theta_1})&\cos({\theta_4+\theta_5+\theta_1+\theta_2})&1&-\cos({\theta_4})\\
\noalign{\medskip}-\cos({\theta_5})&\cos({\theta_5+\theta_1})&-\cos({\theta_5+\theta_1+\theta_2})&-\cos({\theta_4})&1\end {array}
\right)=\\ \small
\left(\begin{array}{c}
          1 \\
          -\cos(\theta_1) \\
          \cos(\theta_1+\theta_2) \\
          \cos(\theta_4+\theta_5) \\
          -\cos(\theta_5)
        \end{array}
 \right)\left(\begin{array}{c}
          1 \\
          -\cos(\theta_1) \\
          \cos(\theta_1+\theta_2) \\
          \cos(\theta_4+\theta_5) \\
          -\cos(\theta_5)
        \end{array}
 \right)^T +\left(\begin{array}{c}
          0 \\
          \sin(\theta_1) \\
          -\sin(\theta_1+\theta_2) \\
          \sin(\theta_4+\theta_5) \\
          -\sin(\theta_5)
        \end{array}
 \right)\left(\begin{array}{c}
          0 \\
          \sin(\theta_1) \\
          -\sin(\theta_1+\theta_2) \\
          \sin(\theta_4+\theta_5) \\
          -\sin(\theta_5)
        \end{array}
 \right)^T.\qquad\qquad\qquad\qquad\end{multline*}
(See page 1674 in  \cite{DickinsonDuerGijbenHildebrand2013b}.)

Finally, we mention that according to \cite[Corollary 5.8]{DickinsonDuerGijbenHildebrand2013a}, if $A\in \cop_5$ with diagonal $\1$  is not SPN, then there exist  $\btheta\in \R^5_+$, $\1^T\btheta< \pi$ and a permutation matrix $P$ such that $P^TAP\ge S(\btheta)$. %%% changed \le to <
We can now prove the special case $n=5$ of  \cite[Theorem 6.4]{Shaked2016}.

\begin{theorem}\label{thm:T5}
$T_5$ is SPN.
\end{theorem}

\begin{proof}
Let $A\in \cop_5$ have $G(A)=T_5$. We may assume that $\diag (A) =\1$.
Let  $i$ be a vertex of degree $2$ in $G(A)$. If the two off-diagonal entries in row $i$ are both positive, then $A$ is
a $4\times 4$ copositive matrix bordered by a nonnegative row and column, and is therefore SPN. If the two off-diagonal entries in row $i$ are both negative,
then $A/A[i]$ is a  $4\times 4$ copositive matrix, and therefore $A/A[i]$ is SPN, and so is $A$. Thus it remains to consider the case that
$A$ has the following pattern, up to permutation of rows and columns:

\begin{equation}A=\left(\begin{array}{ccccc}
               1 & 0 & 0 & - & + \\
               0 & 1 & 0 & - & + \\
               0 & 0 &1 & + & -\\
               - & - & + & 1 & - \\
               +& + &  - &-  & 1
             \end{array}\right)\label{eq:G(A)=T5}\end{equation}
Suppose on the contrary that $A$ is not SPN.
By \cite[Corollary 5.8]{DickinsonDuerGijbenHildebrand2013a} (and since $\diag(A)=\1$),
$P^TAP\ge S(\btheta)$ for some $\btheta\in \R^5_+$ such that \begin{equation}\sum_{i=1}^5\theta_i<\pi, \label{eq:1^Ttheta<pi} \end{equation}
and some permutation matrix $P$.
Let $S=S(\btheta)$. Since $P^TAP\ge S$, $S$ has at most three positive entries above the diagonal.
The matrix $S$ cannot have a row $i$ with no positive entry, because otherwise it would be SPN (since the $4\times 4$ copositive matrix $S/S[i]$ would be SPN).
Thus $S$ has exactly three positive entries above the diagonal, and by the pattern of $A$, these have to be in the same positions as the positive entries of $P^TAP$.
By considering $P^TAP$ instead of $A$, we may assume that $A\ge S$.
Let $B$ be defined by setting $b_{ij}=s_{ij}$ whenever $\{i,j\}$ is an edge in $G(A)$, and $b_{ij}=a_{ij}=0$ otherwise. Then $A\ge B\ge S$, so $B$ is also copositive and not SPN. We assume therefore that $A=B$. That is, $A\ge S$ is not SPN, and  $a_{ij}=s_{ij}$ for every edge $ij$ of $G(A)$, and show that in this case
 there exists $\btheta'\in \R^5_{+}$
such that $\1^T\btheta'=\pi$ and $A\ge S(\btheta')$. Since such $S(\btheta')$ is positive semidefinite, this contradicts the assumption that $A$ is not SPN.

We first note that \begin{equation}\theta_i\le \pi/2 \text{ for every } 1\le i\le 5 .\label{eq:each<=pi/2}\end{equation}
For if, say, $\theta_5>\pi/2$, then
$\sum_{i=1}^4\theta_i<\pi/2$, and we would have  $\cos(\theta_i+\theta_{i\plmod 1})>0$ for $i=1,2, 3$, in addition to $-\cos(\theta_5)>0$,
which would mean that $S$ has at least four positive entries above the diagonal, contradicting the assumption.
So $s_{i,i\plmod 1}\le 0$ for every $i=1, \dots, 5$, and the three positive entries are
all of the form $s_{i,i\plmod 2}=\cos(\theta_i+\theta_{i\plmod 1})$. There exist  $j\ne k$ such that $s_{j,j\plmod 2}\le 0$ and $s_{k,k\plmod 2}\le  0$, that is,
$\theta_i+\theta_{i\plmod 1}\ge \pi/2$, for $i\in \{j, k\}$. Since $\sum_{i\in  \{j, j\plmod 1\}\cup \{k, k\plmod 1\}}\theta_i<\pi$, we must have
$\{j, j\plmod 1\}\cap \{k, k\plmod 1\}\ne \emptyset$. By further permuting row and columns we may assume that
 \begin{equation} \theta_4+\theta_{5}\ge\pi/2 \mbox{  and  } \theta_5+\theta_{1}\ge\pi/2  \label{eq:4+5,5+1}\end{equation}
As the sum of all elements in $\btheta$ is less than $\pi$,
(\ref{eq:4+5,5+1}) implies that $\sum_{i=1}^3\theta_i<\pi/2$ and $\sum_{i=2}^4 \theta_i<\pi/2$.
So the entries  $s_{i,i\plmod 2}=\cos(\theta_i+\theta_{i\plmod 1})$, $i=1,2, 3$,  in positions $(1,3)$, $(2,4)$ and $(3,5)$, are
the three positive entries of $S$. The entries $s_{i,i\plmod 1}=-\cos(\theta_i)$, $i=1, \dots ,4$, are negative. And $s_{15}=-\cos(\theta_5)$, $s_{14}=\cos(\theta_4+\theta_{5})$ and $s_{25}=\cos(\theta_5+\theta_{1})$
are all nonpositive, and are the only entries above the diagonal of $S$ that may be zero.   There is a $3\times 3$ identity principal submatrix  in $A$. None of the
entries $(1,3)$, $(2,4)$, or $(3,5)$ is in this principal submatrix,   so either $A[1,2,5]=I$ or $A[1,4,5]=I$. In the first case $a_{12}=0>s_{12}$ and in
the second $a_{45}=0>s_{45}$. Let $\ell$ be either $1$ or $4$, satisfying $a_{\ell, \ell\plmod 1}=0>s_{\ell, \ell\plmod 1}$. By (\ref{eq:4+5,5+1})
\begin{equation}\sum_{i\ne \ell}\theta_i \ge \pi/2. \label{eq:sumnell}\end{equation}
Let $\btheta'\in \R^5_+$ have
\[\theta'_i=\left\{\begin{array}{ll}\theta_i, \quad & i\ne \ell \\
    \pi-\sum_{i\ne \ell} \theta_i\, ,\quad&  i=\ell  \end{array}
    \right. .
\]
Then \begin{equation}\theta_\ell<\theta'_\ell\le\pi/2.\label{eq:theta'ell}\end{equation}
(The left inequality follows from (\ref{eq:1^Ttheta<pi}), the right one from (\ref{eq:sumnell}).)
Of course, $\theta_i\le \theta'_i$ for every $i$, and $\1^T\btheta'=\pi$. Thus $S(\btheta')$ is positive semidefinite, and $A\ge S(\btheta')$. To assert this latter inequality observe that $a_{\ell,\ell\plmod 1}=0\ge  (S(\btheta'))_{\ell, \ell\plmod 1}$ by (\ref{eq:theta'ell}), and for $i\ne \ell $ we have $a_{i,i\plmod 1} \ge (S(\btheta))_{i, i\plmod 1} = (S(\btheta'))_{i, i\plmod 1}$.
As $\pi \ge \theta'_i+\theta'_{i\plmod 1}\ge \theta_i+\theta_{i\plmod 1}\ge 0$ for every $i$  and $\cos(t)$ decreases on $[0,\pi]$, $a_{i,i\plmod 2} \ge (S(\btheta))_{i, i\plmod 2} \ge (S(\btheta'))_{i, i\plmod 2}$.
\end{proof}

Instead of \cite[Theorem 9.2]{Shaked2016}, which states that every $K_{2,n}$ is SPN, we prove:

\begin{theorem}
If $T_{n+1}$ is SPN, then the complete bipartite graph $K_{2,n}$ is SPN.
\end{theorem}

The proof is essentially the same as the proof of  \cite[Theorem 9.2]{Shaked2016}.

\begin{proof}
By induction on $n$. For $n\le 2$ this holds since every graph on at most $4$ vertices is SPN. If $n>2$ and $T_{n+1}$ is SPN, then its subgraph $T_{n}$ is SPN, and
by the induction hypothesis
$K_{2,n-1}$ is SPN. This implies that each proper subgraph of $K_{2,n}$  is SPN.
Thus we only need to consider the case that $A\in \cop$ has a connected $G_-(A)$. In this case, there exists a vertex $i$ of degree $2$
in $\mathcal{G}(A)$, which is incident with two negative edges. By \cite[Lemma 3.4]{Shaked2016}  $A/A[i]$ is copositive, and since $G(A/A[i])=T_{n+1}$,
it is SPN. Thus $A$ is SPN.
\end{proof}

Since we only know at this point that $T_n$, $n\le 5$, is SPN, we can only deduce that $K_{2,n}$, $n\le 4$ is SPN.
%Finally, we remark that we still conjecture that  \cite[Theorem 6.4]{Shaked2016} is true, i.e., every $T_n$ is SPN.

\def\cprime{$'$}


\begin{thebibliography}{1}

\bibitem{DickinsonDuerGijbenHildebrand2013a}
Peter J.~C. Dickinson, Mirjam D{\"u}r, Luuk Gijben, and Roland Hildebrand.
\newblock Irreducible elements of the copositive cone.
\newblock {\em Linear Algebra and its Applications}, 439(6):1605--1626, 2013.

\bibitem{DickinsonDuerGijbenHildebrand2013b}
Peter J.~C. Dickinson, Mirjam D{\"u}r, Luuk Gijben, and Roland Hildebrand.
\newblock Scaling relationship between the copositive cone and {P}arrilo's
  first level approximation.
\newblock {\em Optimization Letters}, 7(8):1669--1679, 2013.

\bibitem{Hildebrand2012}
Roland Hildebrand.
\newblock The extreme rays of the {$5\times 5$} copositive cone.
\newblock {\em Linear Algebra and its Applications}, 437(7):1538--1547, 2012.

\bibitem{Shaked2016}
Naomi Shaked-Monderer.
\newblock S{PN} graphs: when copositive = {SPN}.
\newblock {\em Linear Algebra and its Applications}, 509:82--113, 2016.

\end{thebibliography}
\end{document}